\documentclass[12pt]{article}
\usepackage{amsmath, pb-diagram}
\usepackage{amssymb}
\usepackage{amscd}
\newenvironment{proof}{\par\noindent{\bf Proof \,}}{$\hfill \Box$\par\bigskip}
\date{\empty}
\newtheorem{thm}{Theorem}[section]
\newtheorem{lem}[thm]{Lemma}
\newtheorem{prop}[thm]{Proposition}
\newtheorem{cor}[thm]{Corollary}

\newtheorem{ex}[thm]{Example}

\begin{document}

\title{Strongly J-clean matrices over 2-projective-free rings}

\author{Marjan Sheibani, Huanyin Chen and Rahman
 Bahmani}

\maketitle

\begin{abstract}
An element $a$ of a ring is strongly J-clean if it is the sum of
an idempotent and an element in the Jacobson radical that
commutes. We characterize the strongly J-clean $2\times 2$
matrices over noncommutative 2-projective-free rings. For a
2-projective-free ring $R$, $A\in M_2(R)$ is strongly J-clean if
and only if $A\in J\big(M_2(R)\big)$, or $I_2-A\in
J\big(M_2(R)\big)$, or $A$ is similar to $\left(
\begin{array}{cc}
0&\lambda\\
1&\mu\end{array} \right)$ where $\lambda\in J(R),\mu\in 1+J(R)$,
and the equation $x^2-x\mu -\lambda=0$ has a root in $J(R)$ and a
root in $1+J(R)$. Strongly J-clean $2\times 2$ matrices over power
series are therefrom investigated. We prove that if $R$ is a
2-projective-free wb-ring then $A(x)\in M_2(R[[x]])$ is strongly
J-clean if and only if $A(0)\in M_2(R)$ is strongly J-clean.

{\bf Keywords:} strongly J-clean matrix; 2-projective-free ring;
quadratic equation, power series.

{\bf 2010 Mathematics Subject Classification:} 15E50, 16U60.
\end{abstract}

\section{Introduction}
\vskip4mm An element $a\in R$ is called {\it strongly J-clean}
({\it clean}) if there exists an idempotent $e\in R$ such that
$a-e\in J(R)$ $(U(R))$ and $ae=ea$. A ring $R$ is {\it strongly
J-clean} ({\it clean}) provided that every element in $R$ is
strongly J-clean (clean). A ring $R$ in uniquely clean if for any
$a\in R$ there exists a unique idempotent $e\in R$ such that
$a-e\in U(R)$, which was introduced by Anderson and
Camillo~\cite{AC}. Nicholson and Zhou proved that a ring $R$ is
uniquely clean if and only if for any $a\in R$ there exists a
unique idempotent $e\in R$ such that $a-e\in J(R)$~\cite[Theorem
20]{NZ}. Evidently, $\{ $ uniquely clean rings $\}\subsetneq\{ $
strongly J-clean rings $\}\subsetneq\{$ strongly clean rings $\}$.
The inclusions are both proper. For instances, The ring $T_2({\Bbb
Z}_2)$ of all $2\times 2$ up triangular matrices over ${\Bbb Z}_2$
is strongly J-clean, while it is not uniquely
clean~\cite[Corollary 16.4.24]{CH}; and that ${\Bbb Z}_3$ is
strongly clean, while it is not strongly J-clean. Thus, the class
of strongly J-clean rings is a medium between those of uniquely
clean rings and strongly clean rings. On the other hand, for an
arbitrary ring $R$, one easily checks that $\left(
\begin{array}{cc}
1&1\\
1&0
\end{array}
\right)\in M_2(R)$ is not strongly J-clean. Hence, it is
attractive to study when a matrix over a ring can be written in
these special forms. In fact, Strong J-cleanness (cleanness) of
$2\times 2$ matrices over commutative local rings are studied by
many authors.~\cite{L} and ~\cite{Y} investigated strongly clean
decompositions of $2\times 2$ matrices over local rings. Recently,
strong cleanness of matrices over a general ring was discussed in
~\cite{CC} and~\cite{D}. Furthermore, strongly $J$-clean matrices
over noncommutative local rings were studied in ~\cite{C}.

Let $R$ be a ring with an identity. We say that a ring $R$ is a
{\it 2-projective-free ring} if every 2-generated projective
R-module is free of constant rank. A ring $R$ is {\it
projective-free} if every finitely generated projective R-module
is free of constant rank. Thus, projective-free rings form a
subclass of 2-projective-free rings. Then we see that division
rings (including fields), local rings, B\'{e}zout domain (i.e.,
domain in which every finitely generated right ideal is principal,
including principal ideal domains, valuation rings, the ring of
entire functions and the ring $\overline{{\Bbb Z}}$ of all
algebraic integers), polynomial rings of a principal ideal domains
are all 2-projective-free. The motivation of this article is to
explore strongly J-clean decompositions of $2\times 2$ matrices
over such new kind of rings.

In Section 2, we shall investigate elementary properties of
2-projective-free rings which will be used in the sequel. In
particular, we prove that every PSF ring is 2-projective-free.
This provides a large class of such rings.

We extend, in Section 3, the main results of strongly J-clean
matrices over local rings in ~\cite{C} to noncommutative
2-projective-free rings. For a 2-projective-free ring $R$, we show
that $A\in M_2(R)$ is strongly J-clean if and only if $A\in
J\big(M_2(R)\big)$, or $I_2-A\in J\big(M_2(R)\big)$, or $A$ is
similar to $\left(
\begin{array}{cc}
0&\lambda\\
1&\mu\end{array} \right)$ where $\lambda\in J(R),\mu\in 1+J(R)$,
and the equation $x^2-x\mu -\lambda=0$ has a root in $J(R)$ and a
root in $1+J(R)$. Strongly J-clean matrices over a commutative
2-projective-free ring are thereby characterized.

Section 4 is concern on strongly J-clean matrices over power
series of 2-projective-free rings. Let $R$ be a ring, and let
$f(x)\in R[[x]]$. If $f(0)\in R$ is optimally J-clean, we prove
that $f(x)\in R[[x]]$ is strongly J-clean. Let $R$ be a ring, and
let $\alpha,\beta\in R$. ${\ell}_{\alpha}-r_{\beta}: R\to R$ is
the group homomorphism given by $x\mapsto \alpha x-x\beta$ for any
$x\in R$. We say that $R$ is a {\it weakly bleached ring} ({\it
wb-ring}, for short) provided that for any $\alpha\in
J(R),\beta\in 1+J(R)$, ${\ell}_{\alpha}-r_{\beta}$ and
${\ell}_{\beta}-r_{\alpha}$ are surjective. For instance, every
commutative ring is a wb-ring. This concept was firstly introduced
only for general local rings (not necessary be commutative)
~\cite{Y}. Let $R$ be a 2-projective-free wb-ring. We shall prove
that $A(x)\in M_2(R[[x]])$ is strongly J-clean if and only if so
is $A(0)\in M_2(R)$ is strongly J-clean.

Throughout, all rings are associative with an identity. $M_n(R)$
will denote the ring of all $n\times n$ full matrices over $R$
with an identity $I_n$. $GL_n(R)$ stands for the $n$-dimensional
general linear group of $R$. Let $M$ be a right module. $end(M)$
and $aut(M)$ stand for the ring of endomorphism and automorphism
of $M$, respectively. We always use $[a,b]$ to stand for the
commutator $ab-ba$ for any $a,b\in R$.

\section{2-Projective-free rings}

\vskip4mm The purpose of this section is to investigate elementary
properties of 2-projective-free rings which will be used in the
sequel. A ring $R$ is called {\it 2-IBN} if $R^m\cong R^n
(m,n=1,2)$ implies that $m=n$. We begin with

\begin{prop} \label{21} A ring is 2-projective-free if and only if $R$ is 2-IBN and every idempotent $2\times 2$ matrix over $R$ admits
a diagonal reduction to $I_2$ or $
\left( \begin{array}{cc} 1&0\\
0&0 \end{array} \right)$.
\end{prop}

\begin{proof} This is similar to that of projective-free rings~\cite[Proposition 2.6]{Co}.
\end{proof}

%\begin{cor} \label{22} Every division ring is 2-projective-free.\end{cor}
%\begin{proof}As every every projective module is free with constant rank over every division ring, so we get the result from the above theorem
%.\end{proof}

\begin{cor}\label{22} Every 2-projective-free ring $R$ has only trivial idempotents.
\end{cor}
\begin{proof} Let $0,1\neq e\in R$ be an idempotent. Then $eR$ is a trivially 2-generated projective module. Hence, $eR$ is free.
Thus, $eR=0$ or $eR\cong R$. Likewise, $(1-e)R=0$ or $(1-e)R\cong
R$. This implies that $eR\cong (1-e)R\cong R$, and so $R\cong
eR\oplus (1-e)R\cong R^2$. By Proposition~\ref{21}, $R$ is 2-IBN.
This gives a contradiction, hence the result.\end{proof}

\vskip4mm As an immediate consequence, we deduce that every
2-projective-free ring is directly finite, i.e., every right or
left invertible element is invertible.

%\begin{cor} Let $R$ be a ring . Then $R$ is 2-projective-free if
%and only if \end{cor}
%\begin{enumerate}
%\item [(1)]{\it $R$ is connected;} \vspace{-.5mm}
%\item [(2)]{\it Every $2\times 2$ idempotent matrix over $R$ admits a diagonal reduction.}
%\end{enumerate}\vspace{-.5mm}  \begin{proof}
%Let $R$ be a 2-projective-free ring then $R$ is connected by
%Corollary~\ref{22} and every $2\times 2$ idempotent matrix over
%$R$ admits a diagonal reduction by definition.

%Conversely assume that $(1)$ and $(2)$ hold, let $A$ be a $2\times
%2$ idempotent matrix over $R$. As every idempotent matrix admits
%diagonal reduction, there exist matrices $U, V \in GL_{2} (R)$
%such that $UAV =B$ is a diagonal matrix. Let $C= V^{-1}U^{-1}$.
%Then $A$ is similar to $BC$. Set $D=BC$. As $B$ is a diagonal
%matrix, so is $BC$, $BC=UAVV^{-1}U^{-1}=UAU^{-1}$ then $A$ is
%similar to diagonal matrix.\end{proof}

\begin{lem} \label{23} Let $I\subseteq J(R)$. If $R/I$ is 2-projective-free, then so is $R$.\end{lem}
\begin{proof} Let $P$ be a 2-generated projective $R$- module then $ P/IP $ is a 2-generated projective $R/I$- module.
As $R/I$ is 2-projective-free, we get $ P/IP \cong (R/I)^{m}\cong
R^m/IR^m$. Since $ I\subseteq J(R)$, we deduce that $ P\cong R^{m}
$, as desired.\end{proof}

\begin{thm} \label{24} Let $R$ be 2-projective-free. Then $R[[x]]$ is 2-projective-free.\end{thm}
\begin{proof}
Let $\varphi :R[[X]]\longrightarrow R$ be a ring homomorphism that
is defined by $\varphi(f) = f(0) $, then $\varphi$ is surjective.
Since invertible elements of $R[[X]]$ are those whose constant
terms are invertible in $R$, we have $ker(\varphi)\subseteq
J(R[[x]]))$. Clearly, $R[[X]]/ker(\varphi)\cong R$. It follows by
Lemma~\ref{23} that $R[[X]]$ is 2-projective-free, as asserted.
\end{proof}

%\begin{cor} \label{24} Every local ring is 2-projective-free.\end{cor}
%\begin{proof}
%Let $R$ be a local ring then $R/J(R)$ is a division ring and we obtain the result by the two last corollaries.
%\end{proof}

An R-module $P$ is {\it stably free} if there exists natural
numbers $m,n$ such that $P\oplus R^m\cong R^n$. It is well known
that the ring ${\Bbb Z}[x,y,x]/(x^2+y^2+z^2-1)$ has stable free
modules which are not free~\cite[Exercise 1.4.23]{R}. A
commutative ring $R$ is called a {\it PSF ring} if each finitely
generated projective module is stably free, i.e., the {\it
Grothendieck group} $K_0(R)={\Bbb Z}[R]$. Every projective-free
rings is PSF. In fact, $\{$ PSF rings $\}\bigcap \{$ Hermite
rings~$\}$ $=\{$~projective-free rings $\}$. The following result
provides a large class of 2-projective-free rings.

\begin{thm} \label{25} Every PSF rings is 2-projective-free.\end{thm}
\begin{proof}
Let $P$ be a $2$-generated projective module over a PSF ring $R$.
Then there exists an epimorphism $\varphi: R^2\to P$, and so
$P\oplus Q\cong R^2$, where $Q=ker(\varphi)$. By hypothesis, $P$
is stable free. Write $P\oplus R^m\cong R^n$. Then $P$ is of
constant rank, i.e., $rank(P)=n-m$. Likewise, $Q$ is of constant
rank. Clearly, $rank(P)+rank(Q)=2$. If $rank(P)=0$, then $P=0$. If
$rank(P)=2$, then $rank(Q)=0$, and so $Q=0$. It follows that
$P\cong R^2$. If $rank(P)=1$, then $P\oplus R^m\cong R^{m+1}$, and
so $P$ is free, as every stably free module of rank 1 over a
commutative ring is free, and we are through.\end{proof}

\begin{ex} Let $R={\Bbb Z}\big[\sqrt{-3}\big]$. It is well known that $K_0(R)\cong {\Bbb Z}, [R]\mapsto 1$, and so $K_0(R)={\Bbb Z}[R]$. Thus, $R$ is PSF.
In terms of Theorem~\ref{25}, $R$ is 2-projective-free. In this
case, $R$ is not principal domain, even it is not a Dedekind
domain ~\cite[Exercise 1.4.21]{R}.
\end{ex}

\section{Strongly J-clean matrices}

The aim of this section is to characterize a single strongly
J-clean matrix over 2-projective-free rings in terms of the
solvability of quadratic equation.

\begin{thm} \label{31} Let $R$
be 2-projective-free. Then $A\in M_2(R)$ is strongly J-clean if
and only if $A\in J\big(M_2(R)\big)$ or $I_2-A\in
J\big(M_2(R)\big)$ or $A$ is similar to a matrix $\left(
\begin{array}{cc}
\alpha&0\\
0&\beta\end{array} \right)$, where $\alpha\in 1+J(R), \beta\in
J(R)$.\end{thm} \begin{proof} $\Longleftarrow$ If $A\in
J\big(M_2(R)\big)$, then $A=0+A$ is strongly J-clean. If $I_2-A\in
J\big(M_2(R)\big)$, then $A=I_2+(A-I_2)$ is strongly J-clean. If
$A$ is similar to a matrix $\left(
\begin{array}{cc}
\alpha&0\\
0&\beta\end{array} \right)$ where $\alpha\in 1+J(R), \beta\in
J(R)$, then there exists some $U\in GL_2(R)$ such that
$$A=U^{-1}\left(
\begin{array}{cc}
1&0\\
0&0\end{array} \right)U+U^{-1}\left(
\begin{array}{cc}
\alpha-1&0\\
0&\beta\end{array} \right)U~\mbox{is stroongly J-clean.}$$

$\Longrightarrow$ By hypothesis, there exists an idempotent $E\in
M_2(R)$ and a $W\in J\big(M_2(R)\big)$ such that $A=E+W$ with
$EW=WE$. Suppose that $A$ and $I_2-A$ are not in
$J\big(M_2(R)\big)$. Since $R$ is 2-projective-free, by virtue of
Proposition~\ref{21}, there exists $U\in GL_2(R)$ such that
$UEU^{-1}=diag(1,0)$. Hence, $UAU^{-1}=\left(
\begin{array}{cc}
1&0\\
0&0\end{array} \right)+UWU^{-1}$. Set $V=(v_{ij}):=UWU^{-1}$. Then
$\left(
\begin{array}{cc}
1&0\\
0&0\end{array} \right)V=V\left(
\begin{array}{cc}
1&0\\
0&0\end{array} \right)$; whence, $v_{12}=v_{21}=0$ and
$v_{11},v_{22}\in J(R)$. Therefore $A$ is similar to $\left(
\begin{array}{cc}
1+v_{11}&0\\
0&v_{22}\end{array} \right)$, which completes the
proof.\end{proof}

\vskip4mm Let $R$ be any ring with $J(R)=0$ (i.e., division ring).
Then $A\in M_2(R)$ is strongly J-clean if and only if $A$ is an
idempotent matrix. Further, we derive

\begin{cor} \label{32} Let $R$ be any 2-projective-free commutative ring with $J(R)=0$ (e.g., ${\Bbb Z}$), and let $A\in M_2(R)$. Then
$A$ is strongly J-clean if and only if $A=0$ or $I_2$, or
$\left( \begin{array}{cc} a&b\\
c&1-a \end{array}\right)$ with $bc=a-a^2$.\end{cor} \begin{proof}
$\Longrightarrow$ In view of Theorem~\ref{31}, $A$ is strongly
J-clean if and only if $A=0$ or $I_2$, or $A$ is similar to
$diag(1,0)$. If $UAU^{-1}=diag(1,0)$ for a $U\in GL_2(R)$, then
$det(A)=0$ and $tr(A)=1$. Thus, $$A=
\left( \begin{array}{cc} a&b\\
c&1-a \end{array}\right)~\mbox{with}~bc=a-a^2.$$

$\Longleftarrow$ One easily checks that $A$ is an idempotent
matrix, and then strongly J-clean.\end{proof}

\begin{lem} \cite[Theorem 2.1]{C}\label{33} Let
$E=end(_RM)$, and let $\alpha\in E$. Then the following are
equivalent:\end{lem}
\begin{enumerate}
\item [(1)]{\it $\alpha$ is strongly J-clean in $E$.}
\vspace{-.5mm}
\item [(2)]{\it $M=P\oplus Q$ where $P$ and $Q$ are $\alpha$-invariant, and $\alpha|_P\in J\big(end(P)\big)$ and
$(1_M-\alpha)|_Q\in J\big(end(Q)\big)$.}
\end{enumerate}

\begin{lem} \label{34} Let $R$ be 2-projective-free, and let $A\in M_2(R)$ be strongly J-clean. Then $A\in
J\big(M_2(R)\big)$ or $I_2-A\in J\big(M_2(R)\big)$ or $A$ is
similar to a matrix $\left(
\begin{array}{cc}
0&\lambda\\
1&\mu\end{array} \right)$, where $\lambda\in J(R),\mu\in
1+J(R)$.\end{lem}
\begin{proof} Suppose that $A,I_2-A\not\in J\big(M_2(R)\big)$. By virtue of Theorem~\ref{31}, we have a $P\in GL_2(R)$ such that
$PAP^{-1}=\left(
\begin{array}{cc}
1+\alpha&0\\
0&\beta\end{array} \right)$, where $\alpha,\beta\in J(R)$. Thus,
we check that
$$UAU^{-1}=\left(
\begin{array}{cc}
0&-(1+\alpha)(1+\alpha-\beta)^{-1}\beta(1+\alpha-\beta)\\
1&(1+\alpha-\beta)^{-1}\beta(1+\alpha-\beta)+(1+\alpha)\end{array}
\right),$$ where $$U=\left(
\begin{array}{cc}
1&-1-\alpha\\
0&1\end{array} \right)\left(
\begin{array}{cc}
1&-1-\alpha\\
0&1\end{array} \right)\left(
\begin{array}{cc}
1&0\\
0&(1+\alpha-\beta)^{-1}\end{array} \right)\left(
\begin{array}{cc}
1&0\\
1&1\end{array} \right).$$ Set
$\lambda=-(1+\alpha)(1+\alpha-\beta)^{-1}\beta(1+\alpha-\beta)$
and $\mu=(1+\alpha-\beta)^{-1}\beta(1+\alpha-\beta)+(1+\alpha)$.
Then $\lambda\in J(R)$ and $\mu\in 1+J(R)$, as desired.\end{proof}

We come now to the main result of this section.

\begin{thm} \label{35} Let $R$ be
2-projective-free. Then $A\in M_2(R)$ is strongly J-clean if and
only if \end{thm}
\begin{enumerate}
\item [(1)]{\it $A\in J\big(M_2(R)\big)$, or} \vspace{-.5mm}
\item [(2)]{\it  $I_2-A\in J\big(M_2(R)\big)$, or}\vspace{-.5mm}
\item [(3)]{\it  $A$ is
similar to $\left(
\begin{array}{cc}
0&\lambda\\
1&\mu\end{array} \right)$ where $\lambda\in J(R),\mu\in 1+J(R)$,
and the equation $x^2-x\mu -\lambda=0$ has a root in $J(R)$ and a
root in $1+J(R)$.}
\end{enumerate}\vspace{-.5mm}  \begin{proof} Suppose that $A\in M_2(R)$ is strongly J-clean, and that
$A,I_2-A\not\in J\big(M_2(R)\big)$. It follows by Lemma~\ref{34}
that $A$ is similar to the matrix $B=\left(
\begin{array}{cc}
0&\lambda\\
1&\mu\end{array} \right)$ where $\lambda\in J(R),\mu\in 1+J(R)$.
Hence, $B\in M_2(R)$ is strongly J-clean. In view of
Lemma~\ref{33}, we have $2R=C\oplus D$ where $(I_2-B)|_C\in
J\big(end(C)\big)$ and $B|_D\in J\big(end(D)\big)$. Thus, $B|_C\in
aut(C)$ and $(I_2-B)|_D\in aut(D)$. Since $R$ is
2-projective-free, it follows by Proposition~\ref{21} that $C$ and
$D$ are free. As $B,I_2-B\not\in J\big(M_2(R)\big)$, we see that
$C,D\cong R$. Assume that $(a,b)$ and $(c,d)$ are bases of $C$ and
$D$, respectively. Then $C=R(a,b), D=R(c,d)$. Then
$$R(a,b)\left(
\begin{array}{cc}
0&\lambda\\
1&\mu\end{array} \right)=R(a,b).$$ Set $\overline{R}=R/J(R)$. Then
$$\overline{R}(\overline{a},\overline{b})\subseteq
\overline{R}(\overline{1},\overline{1}).$$ Similarly,
$$\overline{R}(\overline{c},\overline{d})\subseteq
\overline{R}(\overline{1},\overline{0}).$$ Write
$(\overline{a},\overline{b})=s(\overline{1},\overline{1})$ and
$(\overline{c},\overline{d})=t(\overline{1},\overline{0})$. Then
$$(\overline{1},\overline{1})=z(\overline{a},\overline{b})+z'(\overline{c},\overline{d})=zs(\overline{1},\overline{1})+z't(\overline{1},\overline{0}).$$
This implies that $1-zs\in J(R)$, and so $s\in R$ is left
invertible. Hence, $s\in U(R)$, as $R$ is directly finite.
Clearly, $a-s,b-s\in J(R)$, and so $1-a^{-1}b\in J(R)$.
$C=R(a,b)=R(1,\alpha),$ where $\alpha=a^{-1}b\in 1+J(R)$.
Analogously, $D=R(1,\beta)$, where $\beta=c^{-1}d\in J(R)$. As $C$
is $B$-invariant, we see that $$(1,\alpha)\left(
\begin{array}{cc}
0&\lambda\\
1&\mu\end{array} \right)=r(1,\alpha)$$ for some $r\in R$. It
follows that $\alpha=r$ and $\lambda+\alpha\mu=r\alpha$, and
therefore $\alpha^2-\alpha\mu -\lambda=0$, i.e.,
$x^2-x\mu-\lambda=0$ has a root $\alpha\in 1+J(R)$. Likewise, this
equation has a root $\beta\in J(R)$, as desired.

Conversely, if $(1)$ or $(2)$ holds then $A\in M_2(R)$ is strongly
J-clean, and so we assume $(3)$ holds. As strong J-cleanness is
invariant under similarity, we will suffice to check if $B=\left(
\begin{array}{cc}
0&\lambda\\
1&\mu
\end{array}
\right)$ is strongly J-clean. By hypothesis, the equation
$x^2-x\mu -\lambda=0$ has roots $c\in J(R)$ and $d\in 1+J(R)$.
Then $c^2-c\mu -\lambda=0$ and $d^2-d\mu -\lambda=0$. Choose
$C=R(1, c)$ and $D=R(1, d)$. Since $$(1,c)\left(
\begin{array}{cc}
0&\lambda\\
1&\mu
\end{array}
\right)=c(1,c)\in C,$$ $C$ is $B$-invariant. Similarly, $D$ is
$B$-invariant. If $r(1,c)=s(1,d)\in C\bigcap D$, then $r=s$ and
$rc=sd$; hence, $r(c-d)=0$. Since $c-d\in U(R)$, we get $r=0$.
Thus, $C\bigcap D=0$. Let $(a,b)\in 2R$. Choose
$s=(b-ac)(d-c)^{-1}$ and $r=a-s$. Then $(a,b)=r(1,c)+s(1,d)\in
C\oplus D$. Hence, $2R=C\oplus D$. Let $\gamma\in end(C)$. Then
$$\begin{array}{llll} 1_C-B|_C\gamma:
&C&\to& C;\\
&r(1,c)&\mapsto&r(1,c)-rc(1,c)\gamma.
\end{array}
$$ Write $(1,c)\gamma =b(1,c)$ for a $b\in R$. If
$\big(r(1,c)\big)\big(1_C-B|_C\gamma\big)=0$, then
$r(1,c)-rcb(1,c)=0$; hence, $r(1-cb)(1,c)=0$. It follows from
$c\in J(R)$ that $r=0$, and so $r(1,c)=0$. Thus, $1_C-B|_C\gamma$
is monomorphic. For any $r(1,c)\in C$, we see that
$$\big(r(1-cb)^{-1}(1,c)\big)\big(r(1,c)\big)\big(1_C-B|_C\gamma\big)=r(1,c).$$
This implies that $1_C-B|_C\gamma$ is epimorphic. As a result,
$1_C-B|_C\gamma$ is isomorphic. We infer that $B|_{C}\in
J\big(end(C)\big).$ Similarly, $\big(I_2-B\big)|_{D}\in
J\big(end(D)\big).$ In light of Lemma~\ref{34}, $B\in M_2(R)$ is
strongly J-clean, as needed.\end{proof}

\vskip4mm A matrix $A\in M_2(R)$ is {\it cyclic} if there exists a
column $\alpha$ such that $(\alpha,A\alpha)\in GL_2(R)$. For
instance, $\left(
\begin{array}{cc}
*&*\\
u&*
\end{array}
\right)\in M_2(R)$ is cyclic for any $u\in U(R)$.

\begin{cor} \label{36} Let $R$
be a 2-projective-free ring, and let $A\in M_2(R)$. If $R$ is
commutative, then $A$ is strongly J-clean if and only if\end{cor}
\begin{enumerate}
\item [(1)]{\it $A\in J\big(M_2(R)\big)$, or} \vspace{-.5mm}
\item [(2)]{\it  $I_2-A\in J\big(M_2(R)\big)$, or}\vspace{-.5mm}
\item [(3)]{\it  $A$ is cyclic and $x^2-tr(A)x+det(A)=0$ has a root in $J(R)$ and a
root in $1+J(R)$.}
\end{enumerate}\vspace{-.5mm} \begin{proof} Suppose that $A$ is strongly J-clean. If
$A, I_2-A\not\in J\big(M_2(R)\big)$, then $A$ is similar to
$\left(
\begin{array}{cc}
0&\lambda\\
1&\mu\end{array} \right)$ where $\lambda\in J(R),\mu\in 1+J(R)$,
and the equation $x^2-x\mu -\lambda=0$ has a root in $J(R)$ and a
root in $1+J(R)$, by Theorem~\ref{35}. In view of ~\cite[Lemma
7.4.6]{CH}, $A$ is cyclic. As $R$ is commutative, we see that
$tr(A)=\mu$ and $det(A)=-\lambda$, and so $x^2-tr(A)x+det(A)=0$
has a root in $J(R)$ and a root in $1+J(R)$.

Conversely, if $A\in J\big(M_2(R)\big)$ or $I_2-A\in
J\big(M_2(R)\big)$ then $A$ is strongly J-clean. We now assume
that $A$ is cyclic and $x^2-tr(A)x+det(A)=0$ has a root $\alpha$
in $J(R)$ and a root $\beta$ in $1+J(R)$. In view of ~\cite[Lemma
7.4.6]{CH}, $A$ is isomorphic to a companion matrix $\left(
\begin{array}{cc}
0&\lambda\\
1&\mu\end{array} \right)$. This shows that $\mu=tr(A)$ and
$det(A)=-\lambda$. Since
$$\alpha^2-tr(A)\alpha+det(A)=0~\mbox{and}~\beta^2-tr(A)\beta+det(A)=0,$$
we get $tr(A)=\alpha+\beta$ and $det(A)=\alpha\beta$. Hence,
$\mu=\alpha+\beta\in 1+J(R)$ and $\lambda=-\alpha\beta\in J(R)$.
Therefore we complete the proof, by Theorem~\ref{35}.\end{proof}

Let $R$ be a PSF ring, and let $A\in M_2(R)$. It follows from
Corollary~\ref{36} that $A$ is strongly J-clean if and only if
$A\in J\big(M_2(R)\big)$, or $I_2-A\in J\big(M_2(R)\big)$, or $A$
is cyclic and $x^2-tr(A)x+det(A)=0$ has a root in $J(R)$ and a
root in $1+J(R)$. The following is a dual of ~\cite[Theorem
2.5]{C}.

\begin{cor} \label{37} Let $R$ be
a local ring. Then $A\in M_2(R)$ is strongly J-clean if and only
if \end{cor}
\begin{enumerate}
\item [(1)]{\it $A\in J\big(M_2(R)\big)$, or} \vspace{-.5mm}
\item [(2)]{\it  $I_2-A\in J\big(M_2(R)\big)$, or}\vspace{-.5mm}
\item [(3)]{\it  $A$ is
similar to $\left(
\begin{array}{cc}
0&\lambda\\
1&\mu\end{array} \right)$ where $\lambda\in J(R),\mu\in 1+J(R)$,
and the equation $x^2-\mu x -\lambda=0$ has a root in $J(R)$ and a
root in $1+J(R)$.}
\end{enumerate}\vspace{-.5mm}  \begin{proof} Suppose that $A\in M_2(R)$ is strongly
J-clean. Then there exist an idempotent $E\in M_2(R)$ and $W\in
M_2(J(R))$ such that $A=E+W$ and $AE=EA$. Hence,
$(A^o)^T=(E^o)^T+(W^o)^T$. One easily checks that $(E^o)^T\in
M_2(R^{op})$ is an idempotent matrix and $(W^o)^T\in
M_2(J(R^{op}))$. Furthermore, $(A^o)^T(E^o)^T=(E^o)^T(A^o)^T$.
Therefore, $(A^{o})^T\in M_2(R^{op})$ is strongly J-clean.
Clearly, $R^{op}$ is local, and then it is 2-projective-free.
Applying Theorem~\ref{36} to $(A^{o})^T\in M_2(R^{op})$. Then
$(A^{o})^T\in J\big(M_2(R^{op})\big)$, or $I_2^{o}-(A^{o})^T\in
J\big(M_2(R)\big)$, or $(A^{o})^T$ is similar to $\left(
\begin{array}{cc}
0^{o}&\lambda^{o}\\
1^{o}&\mu^{o}\end{array} \right)$ where $\lambda^{o}\in
J(R^{op}),\mu\in 1^{o}+J(R^{op})$, and the equation $x^2-x\mu^{o}
-\lambda^{o}=0$ has a root in $J(R^{op})$ and a root in
$1^{o}+J(R^{op})$. Thus, $A\in J\big(M_2(R)\big)$, or $I_2-A\in
J\big(M_2(R)\big)$, or $A$ is similar to $\left(
\begin{array}{cc}
0&\lambda\\
1&\mu\end{array} \right)$ where $\lambda\in J(R),\mu\in 1+J(R)$,
and the equation $x^2-\mu x-\lambda=0$ has a root in $J(R)$ and a
root in $1+J(R)$.

The converse is proved by a similar route.\end{proof}

\section{Power Series Rings}

\vskip4mm This section is concern on strongly J-clean matrices
over power series rings. We say that an element $a\in R$ is {\it
optimally J-clean} provided that there exists an idempotent $e\in
R$ such that $a-e\in J(R)$ and $ae=ea$, and that for any $b\in R$,
there exists $c\in R$ such that $[a,c]=[e,b]$. A ring $R$ is {\it
optimally J-clean} provided that every element in $R$ is optimally
J-clean. Every uniquely clean ring is optimally J-clean. In view
of ~\cite[Theorem 20]{NZ}, uniquely clean rings are strongly
J-clean. Further, they have the property that all idempotents are
central. Hence, the additional condition for optimally clean is
automatically satisfied (just take $c=a$). Strongly clean
power-series over noncommutative rings have ever been studied by
Shifflet~\cite{Sh}. We now establish strong J-cleanness of power
series over a general ring.

\begin{lem} \label{41} Let $R$ be a
ring, and let $a\in R$. Then the following are
equivalent:\end{lem}
\begin{enumerate}
\item [(1)]{\it $a\in R$ is optimally J-clean.}
\vspace{-.5mm}
\item [(2)]{\it There exists an idempotent $e\in R$ such that
$a-e\in J(R)$ and $ae=ea$, and that for any $b\in R$, there exists
$c\in eR(1-e)+(1-e)Re$ such that $[a,c]+[e,b]=0$.} \vspace{-.5mm}
\end{enumerate} \begin{proof} $(1)\Rightarrow (2)$ Since $a\in R$ is optimally
J-clean, there exists an idempotent $e\in R$ such that $a-e\in
J(R)$ and $ae=ea$, and that for any $b\in R$, there exists $c\in
R$ such that $[a,c]=[e,b]$. It is easy to check that
$$\begin{array}{lcl}
[a,ec(1-e)+(1-e)ce]&=&[a,ec(1-e)]+[a,(1-e)ce]\\
&=&e[a,c](1-e)+(1-e)[a,c]e\\
&=&e[e,b](1-e)+(1-e)[e,b]e\\
&=&[e,b], \end{array}$$ and therefore
$$[a,-ec(1-e)-(1-e)ce]+[e,b]=0.$$

$(2)\Rightarrow (1)$ There exists an idempotent $e\in R$ such that
$a-e\in J(R)$ and $ae=ea$, and that for any $b\in R$, there exists
$c\in eR(1-e)+(1-e)Re$ such that $[a,c]+[e,b]=0$. Choose $c'=-c$.
Then $[a,c']=[e,b]$, as required.
\end{proof}

\vskip4mm The following two lemmas are taken from a thesis and not
a published paper~\cite{Sh}, and so we include simple proofs to
indicate how to get these results

\begin{lem} ~\cite[Lemma 3.2.1]{Sh} \label{42} Let $R$ be a
ring, and let $n\geq 2$. If $e_0=e_0^2\in R$ and
$e_k(1-e_0)=\sum\limits_{i=0}^{k-1}e_ie_{k-i} (0< k< n)$, then
$e_0\big(\sum\limits_{i=1}^{n-1}e_ie_{n-i}\big)=\big(\sum\limits_{i=1}^{n-1}e_ie_{n-i}\big)e_0.$\end{lem}
\begin{proof} Straightforward.\end{proof}

\begin{lem} ~\cite[Theorem 3.2.2]{Sh} \label{43} Let $R$ be a
ring, and let $n\geq 2$. If $e_0=e_0^2\in R,
e_k(1-e_0)=\sum\limits_{i=0}^{k-1}e_ie_{k-i}$ and
$[r_0,e_k]+[r_1,e_{k-1}]+\cdots +[r_k,e_0]=0$ for all $0< k< n$.
Then
$$\big[r_0,\sum\limits_{i=1}^{n-1}e_ie_{n-i}\big]=(1-e_0)\big(\sum\limits_{i=1}^{n-1}[e_i,r_{n-i}]\big)-
\big(\sum\limits_{i=1}^{n-1}[e_i,r_{n-i}]\big)e_0.$$\end{lem}
\begin{proof} Let
$t_n=\sum\limits_{i=1}^{n-1}[e_i,r_{n-i}]$,
$\alpha_k=e_1r_{k-1}+\cdots +e_{k-1}r_1$ and
$\beta_k=r_{k-1}e_1+\cdots +r_1e_{k-1}$. Then
$\alpha_ne_0+e_0\alpha_n=\alpha_n-\gamma_n+e_{n-1}r_0$ where
$\gamma_n=\big(e_1r_{n-2}+e_2r_{n-3}+\cdots +e_{n-1}r_0\big)e_1
+\big(e_1r_{n-3}+e_2r_{n-4}+\cdots +e_{n-2}r_0\big)e_2+\cdots
+\big(e_1r_0\big)e_{n-1}$. Likewise,
$\beta_ne_0+e_0\beta_n=\beta_n-\lambda_n+r_0e_{n-1}, $ where
$\lambda_n=e_1\big(r_{n-2}e_1+r_{n-3}e_2+\cdots +r_0e_{n-1}\big)
+e_2\big(r_{n-3}e_1+r_{n-4}e_2+\cdots +r_0e_{n-2}\big)+\cdots
+e_{n-1}\big(r_0e_1\big)$. Further, we verify that
$\gamma_n=\lambda_n.$ Therefore
$e_0t_n+t_ne_0=t_n+e_{n-1}r_0-r_0e_{n-1},$ and so
$e_0\big(\sum\limits_{i=1}^{n-1}[e_i,r_{n-i}]\big)+
\big(\sum\limits_{i=1}^{n-1}[e_i,r_{n-i}]\big)e_0=\sum\limits_{i=1}^{n-1}[e_i,r_{n-i}]-\big[r_0,\sum\limits_{i=1}^{n-1}e_ie_{n-i}\big]$,
as needed.\end{proof}

\vskip4mm In ~\cite[Theorem 3.2.2]{Sh}, Shifflet characterized
strongly clean power series in terms of optimal cleanness.
Following a similar route, we now modify Shifflet's method and
apply to strongly J-clean power series by means of optimal
J-cleanness.

\begin{thm} \label{44} Let $R$ be a
ring, and let $f(x)\in R[[x]]$. If $f(0)\in R$ is optimally
J-clean, then $f(x)\in R[[x]]$ is strongly
J-clean.\end{thm}\begin{proof} Write
$f(x)=\sum\limits_{i=0}^{\infty}r_ix^i$. Then we can find an
idempotent $e_0$ such that $r_0=e_0+(r_0-e_0)$ is an optimally
J-clean decomposition of $r_0$. In view of Lemma~\ref{41}, there
exists some $e_1\in (1-e_0)Re_0+e_0R(1-e_0)$ such that
$[r_0,e_1]+[e_0,r_1]=0$. Clearly, $e_1=e_0e_1+e_1e_0$. We shall
prove that there exist $e_2,\cdots ,e_k,\cdots \in R$ such that
$$e_k=e_0e_k+e_1e_{k-1}+\cdots
+e_ke_0~\mbox{and}~[r_0,e_k]+[r_1,e_{k-1}]+\cdots +[r_k,e_0]=0.$$
Assume that this is true for all $1\leq k\leq n-1$. Set
$f_n=(1-2e_0)(e_1e_{n-1}+e_2e_{n-2}+\cdots +e_{n-1}e_1)$ and
$s_n=r_n+\big[e_0, [e_1,r_{n-1}]+[e_2,r_{n-2}]+\cdots
+[e_{n-1},r_{1}]\big]$. By virtue of Lemma~\ref{41}, we have some
$g_n\in (1-e_0)Re_0+e_0R(1-e_0)$ such that $[r_0,g_n]=[e_0,s_n]$.
Let $e_n=f_n+g_n$. In light of Lemma~\ref{42}, analogously to
~\cite[Theorem 3.2.2]{Sh}, we obtain
$$\sum\limits_{i=1}^{n-1}e_ie_{n-i}=(1-e_0)e_n-e_ne_0.$$
Thus, $e_n=\sum\limits_{i=1}^{n}e_ie_{n-i}$. Furthermore, that
$$\begin{array}{lll}
[r_0,f_n]&=&\big[r_0,(1-e_0)(\sum\limits_{i=1}^{n-1}e_ie_{n-i})\big]-\big[r_0,(\sum\limits_{i=1}^{n-1}e_ie_{n-i})e_0\big]\\
&=&(1-e_0)\big[r_0,(\sum\limits_{i=1}^{n-1}e_ie_{n-i})\big](1-e_0)-e_0\big[r_0,(\sum\limits_{i=1}^{n-1}e_ie_{n-i})\big]e_0.
\end{array}$$
By using Lemma~\ref{43}, we have
$$\big[r_0,\sum\limits_{i=1}^{n-1}e_ie_{n-i}\big]=(1-e_0)\big(\sum\limits_{i=1}^{n-1}[e_i,r_{n-i}]\big)-
\big(\sum\limits_{i=1}^{n-1}[e_i,r_{n-i}]\big)e_0,$$ and then
$$[r_0,f_n]=(1-e_0)\big(\sum\limits_{i=1}^{n-1}[e_i,r_{n-i}]\big)(1-e_0)+e_0\big(\sum\limits_{i=1}^{n-1}[e_i,r_{n-i}]\big)e_0.$$ Moreover,
$$\begin{array}{lll}
[r_0,g_n]&=&[e_0,s_n]\\
&=&[e_0,r_n]+\big[e_0,[e_0,\sum\limits_{i=1}^{n-1}[e_i,r_{n-i}]]\big]\\
&=&[e_0,r_n]+e_0\big(\sum\limits_{i=1}^{n-1}[e_i,r_{n-i}]\big)(1-e_0)+(1-e_0)\big(\sum\limits_{i=1}^{n-1}[e_i,r_{n-i}]\big)e_0.
\end{array}$$
Thus, we get $$\begin{array}{ll}
&[r_0,e_n]\\
=&[r_0,f_n]+[r_0,g_n]\\
=&[e_0,r_n]+e_0\big(\sum\limits_{i=1}^{n-1}[e_i,r_{n-i}]\big)(1-e_0)+(1-e_0)\big(\sum\limits_{i=1}^{n-1}[e_i,r_{n-i}]\big)e_0\\
+&(1-e_0)
\big(\sum\limits_{i=1}^{n-1}[e_i,r_{n-i}]\big)(1-e_0)+e_0\big(\sum\limits_{i=1}^{n-1}[e_i,r_{n-i}]\big)(1-e_0)\\
=&\sum\limits_{i=0}^{n-1}[e_i,r_{n-i}]; \end{array}$$ hence that
$\sum\limits_{i=0}^{n}[r_i,e_{n-i}]=0$. By induction, the claim is
true. Thus,
$\sum\limits_{i=0}^{\infty}e_ix^i=\big(\sum\limits_{i=0}^{\infty}e_ix^i\big)^2\in
R[[x]]$ and
$f(x)\big(\sum\limits_{i=0}^{\infty}e_ix^i\big)=\big(\sum\limits_{i=0}^{\infty}e_ix^i\big)f(x)$.
Since $f(0)-e(0)\in J(R)$, we see that
$f(x)-\sum\limits_{i=0}^{\infty}e_ix^i\in J\big(R[[x]]\big)$.
Therefore $f(x)\in R[[x]]$ is strongly J-clean, as
asserted.\end{proof}

\begin{cor} Let $R$ be an abelian ring, and let $f(x)\in R[[x]]$. If $f(0)\in R$ is strongly J-clean, then $f(x)\in R[[x]]$ is strongly J-clean.\end{cor}\begin{proof}
Suppose $f(0)\in R$ is strongly J-clean. Then there exists an
idempotent $e\in R$ such that $f(0)-e\in J(R)$ and $f(0)e=ef(0)$.
For any $b\in R$, we choose $c=0\in eR(1-e)+(1-e)Re$. Then
$[f(0),c]+[e,b]=0$; hence that $f(0)$ is J-optimally clean.
Therefore $f(x)\in R[[x]]$ is strongly J-clean, in terms of
Theorem~\ref{44}.\end{proof}

\begin{cor} Let $R$ be a ring, and let $f(x)\in R[[x]]$. Then the following are
equivalent:\end{cor}
\begin{enumerate}
\item [(1)]{\it $f(0)\in R$ is optimally J-clean.}
\vspace{-.5mm}
\item [(2)]{\it $f(x)\in R[[x]]$ is optimally J-clean.}
\vspace{-.5mm}
\end{enumerate} \begin{proof} $(1)\Rightarrow (2)$ In view of
Theorem~\ref{44}, $f(x)\in R[[x]]$ is strongly J-clean. Hence,
there exists an idempotent $e(x)\in R[[x]]$ such that
$w(x):=f(x)-e(x)\in J(R[[x]])$ and $f(x)e(x)=e(x)f(x)$. Thus,
$f(x)=\big(1-e(x)\big)+\big(2e(x)-1+w(x)\big)$. As
$\big(2e(x)-1\big)^2=1$, we see that
$\big(2e(x)-1+w(x)\big)=(2e(x)-1)\big(1+(2e(x)-1)w(x)\big)\in
U(R[[x]])$. By virtue of ~\cite[Theorem 3.3.2]{Sh}, $f(x)\in
R[[x]]$ is optimally clean. For any $b(x)\in R[[x]]$, there exists
$c(x)\in R[[x]]$ such that $[f(x),-c(x)]=[1-e(x),b(x)]$. This
implies that $[f(x),c(x)]=[e(x),b(x)]$. Therefore $f(x)\in R[[x]]$
is optimally J-clean, as desired.

$(2)\Rightarrow (1)$ This is obvious.\end{proof}

\begin{lem} \label{45} Let $R$ be a
ring, and let $u\in U(R)$. If $uau^{-1}\in R$ is J-optimally
clean, then so is $a$ in $R$.\end{lem}\begin{proof} Since
$uau^{-1}\in R$ is J-optimally clean, there exists an idempotent
$e\in R$ such that $uau^{-1}-e\in J(R)$ and
$(uau^{-1})e=e(uau^{-1})$, and that for any $b\in R$, there exists
a $c\in R$ such that $[uau^{-1},c]=[e,u^{-1}bu]$. Thus,
$a-u^{-1}eu\in J(R)$ and $a(u^{-1}eu)=(u^{-1}eu)a$. Furthermore,
$[a, uau^{-1}]=[u^{-1}eu, b]$. Accordingly, $a\in R$ is
J-optimally clean.\end{proof}

\begin{lem} \label{46} Let $R$ be
a 2-projective-free wb-ring, and let $A\in M_2(R)$. Then the
following are equivalent:\end{lem}
\begin{enumerate}
\item [(1)]{\it $A\in M_2(R)$ is optimally J-clean.}
\vspace{-.5mm}
\item [(2)]{\it $A\in M_2(R)$ is strongly J-clean.} \vspace{-.5mm}
\end{enumerate}\begin{proof} $(1)\Rightarrow (2)$ This is obvious.

$(2)\Rightarrow (1)$ Suppose that $A\in M_2(R)$ is strongly
J-clean. In view of Theorem~\ref{31}, $A\in J\big(M_2(R)\big)$, or
$I_2-A\in J\big(M_2(R)\big)$, or $A$ is similar to a diagonal
matrix $diag(\alpha,\beta)$ with $\alpha\in J(R),\beta\in 1+J(R)$.
If $A\in J\big(M_2(R)\big)$, then $A-0= A\in J\big(M_2(R)\big)$.
For any $B\in M_2(R)$, we have $[A,0]=[0,B]$, and so $A\in M_2(R)$
is optimally J-clean. If $I_2-A\in J\big(M_2(R)\big)$, then
$A-I_2\in J\big(M_2(R)\big)$. For any $B\in M_2(R)$, we see that
$[A,0]=[I_2,B]$, and then $A\in M_2(R)$ is optimally J-clean.
Hence, we have a $P\in GL_2(R)$ such that
$PAP^{-1}=diag(\alpha,\beta)$, where $\alpha\in J(R)$ and
$\beta\in 1+J(R)$. For any $B=(b_{ij})\in M_2(R)$, by hypothesis,
we have $c_1,c_2\in R$ such that $$\alpha
c_1-c_1\beta=-b_{12}~\mbox{and}~\beta c_2-c_2\alpha=b_{21}.$$ Set
$C=\left(
\begin{array}{cc}
0&c_1\\
c_2&0
\end{array}
\right)$. Then
$$\big[diag(\alpha,\beta),C\big]=\left(
\begin{array}{cc}
0&-b_{12}\\
b_{21}&0
\end{array}
\right)=\big[diag(0,1),B\big].$$ One easily checks that
$diag(\alpha,\beta)-diag(0,1)\in J\big(M_2(R)\big)$. Therefore
$PAP^{-1}$ is optimally J-clean. In light of Lemma~\ref{45}, $A\in
M_2(R)$ is J-optimally clean.\end{proof}

\vskip4mm Strongly J-clean matrices over local rings were studied
in ~\cite[Theorem 3.2]{CH}, but the proof there depends on the
local property. We now have at our disposal all the information
necessary to prove the following.

\begin{thm} \label{47} Let $R$ be
a 2-projective-free wb-ring. Then the following are
equivalent:\end{thm}
\begin{enumerate}
\item [(1)]{\it $A(x)\in M_2(R[[x]])$ is strongly J-clean.}
\vspace{-.5mm}
\item [(2)]{\it $A(0)\in M_2(R)$ is strongly J-clean.} \vspace{-.5mm}
\end{enumerate}\begin{proof} $(1)\Rightarrow (2)$ By hypothesis,
there exists an idempotent $E(x)\in M_2(R[[x]])$ such that
$A(x)-E(x)\in J\big(M_2(R[[x]])\big)$ and $E(x)A(x)=A(x)E(x)$.
This implies that $A(0)-E(0)\in J\big(M_2(R)\big)$ and
$A(0)E(0)=E(0)A(0)$. Therefore, proving $(2)$.

$(2)\Rightarrow (1)$ Since $A(0)\in M_2(R)$ is strongly J-clean,
by virtue of Lemma~\ref{46}, $A(0)\in M_2(R)$ is J-optimally
clean. According to Theorem~\ref{44}, $A(x)\in M_2(R[[x]])$ is
strongly J-clean.\end{proof}

\begin{cor} \label{48} Let $R$ be
a 2-projective-free ring. If $J(R)$ is nil, then the following are
equivalent:\end{cor}
\begin{enumerate}
\item [(1)]{\it $A(x)\in M_2(R[[x]])$ is strongly J-clean.}
\vspace{-.5mm}
\item [(2)]{\it $A(0)\in M_2(R)$ is strongly J-clean.} \vspace{-.5mm}
\end{enumerate}\begin{proof} Let $\alpha\in 1+J(R), \beta\in J(R)$. Write $\beta^n=0$. Choose
$\varphi={\ell}_{\alpha^{-1}}+{\ell}_{\alpha^{-2}}r_{\beta}+\cdots
+{\ell}_{\alpha^{-n}}r_{\beta^{n-1}}: R\to R$. For any $r\in R$,
one easily checks that
$$\begin{array}{ll}
&\big({\ell}_{\alpha}-r_{\beta}\big)\varphi
(r)\\
=&\big({\ell}_{\alpha}-r_{\beta}\big)\big(\alpha^{-1}r+\alpha^{-2}r\beta+\cdots +\alpha^{-n}r\beta^{n-1}\big)\\
=&\big(r+\alpha^{-1}r\beta+\cdots
+\alpha^{-n+1}r\beta^{n-1}\big)-\big(\alpha^{-1}r\beta+\alpha^{-2}r\beta^2+\cdots
+\alpha^{-n+1}r\beta^{n-1}\big)\\
=&r. \end{array}$$ Thus,
$\big({\ell}_{\alpha}-r_{\beta}\big)\varphi=1_R$. Thus,
${\ell}_{\alpha}-r_{\beta}: R\to R$ is surjective. Likewise,
${\ell}_{\beta}-r_{\alpha}: R\to R$ is surjective. Hence, $R$ is a
wb-ring. This completes the proof, by Theorem~\ref{47}.\end{proof}

\begin{cor} \label{49} Let $R$ be a PSF ring. Then the following are equivalent:\end{cor}
\begin{enumerate}
\item [(1)]{\it $A(x)\in M_2(R[[x]])$ is strongly J-clean.}
\vspace{-.5mm}
\item [(2)]{\it $A(0)\in M_2(R)$ is strongly J-clean.} \vspace{-.5mm}
\end{enumerate}\begin{proof} Since $R$ is PSF, it follows by Theorem~\ref{25} that $R$ is 2-projective-free.
On the other hand, $R$ is a wb-ring, as it is commutative.
According to Theorem~\ref{47}, we establish the result.\end{proof}

\vskip15mm\hspace{-2.0em} {\Large\bf Acknowledgements}

\vskip4mm The authors are grateful to the referee for his/her
useful suggestions which correct several errors and improve many
statements, and make the new version more clearer.

\vskip10mm

Marjan Sheibani Abdolyousefi

Faculty of Mathematics, Statistics and Computer Science

Semnan University, Semnan, Iran

Email: m.sheibani1@gmail.com\\

Corresponding author.

Huanyin Chen

Department of Mathematics

Hangzhou Normal University

Hangzhou, 310036, China

Email: huanyinchen@aliyun.com\\

Rahman Bahmani Sangesari

Faculty of Mathematics, Statistics and Computer Science

Semnan University, Semnan, Iran

rbahmani@semnan.ac.ir\\

\end{document}